\documentclass{amsart} 

\usepackage{amsmath, amsfonts, amssymb, amsthm, mathtools}
\usepackage[T1]{fontenc}
\usepackage[scr = dutchcal, bfscr]{mathalfa} 
\usepackage{ae}
\usepackage{aecompl}

\usepackage{adjustbox}
\usepackage{enumitem}

\usepackage{subcaption}
\usepackage{graphicx}

\usepackage{hyperref}
\usepackage[capitalise, noabbrev]{cleveref}
\hypersetup{
	colorlinks = true,
	citecolor  = blue,
    linkcolor  = blue,
    urlcolor   = blue,
    pdfauthor = {Trung Chau, Richie Sheng, Deborah Wooton},
    pdftitle = {Homological invariants  of edge ideals of weighted oriented graphs}
}

\newtheorem{theorem}{Theorem}[section]
\newtheorem{corollary}[theorem]{Corollary}
\newtheorem{lemma}[theorem]{Lemma}
\newtheorem{proposition}[theorem]{Proposition}

\numberwithin{equation}{section}

\theoremstyle{definition}
\newtheorem{definition}[theorem]{Definition}

\newtheorem{example}[theorem]{Example}

\newtheorem{chunk}[theorem]{}

\newcommand{\mc}{\mathcal}
\newcommand{\mf}{\mathfrak}

\newcommand{\NN}{\mathbb{N}}

\renewcommand{\phi}{\varphi}

\newcommand{\depth}{\operatorname{depth}}

\newcommand{\height}{\operatorname{ht}}

\newcommand{\mingens}{\operatorname{Mingens}}
\newcommand{\lcm}{\operatorname{lcm}}

\newcommand{\pdim}{\operatorname{pdim}}
\newcommand{\reg}{\operatorname{reg}}
\newcommand{\pd}{\operatorname{pd}}

\newcommand{\ddUnw}[1]{\mathcal{DD}({#1})}
\newcommand{\ddrUnw}[1]{\mathcal{DDR}({#1})}
\newcommand{\dd}[1]{\mathscr{DD}({#1})}
\newcommand{\ddr}[1]{\mathscr{DDR}({#1})}
\newcommand{\bptUnw}[1]{\mathcal{PR}^{BPT}({#1})}
\newcommand{\treeUnw}[1]{\mathcal{PR}^{TREE}({#1})}
\newcommand{\bpt}[1]{\mathscr{PR}^{BPT}({#1})}
\newcommand{\tree}[1]{\mathscr{PR}^{TREE}({#1})}

\newcommand{\w}{\mathbf{w}}

\title{Homological invariants of edge ideals of weighted oriented graphs}

\author[Chau]{Trung Chau}
\address{Chennai Mathematical Institute, H1 SIPCOT IT Park, Siruseri, Kelambakkam, \newline\indent India}
\email{chauchitrung1996@gmail.com}

\author[Sheng]{Richie Sheng}
\address{Department of Mathematics, University of Utah, Salt Lake City, UT, USA}
\email{u1415944@utah.edu}

\author[Wooton]{Deborah Wooton}
\address{Department of Mathematics, University of Utah, Salt Lake City, UT, USA}
\email{deborah.wooton@utah.edu}

\subjclass{13D02, 05E40, 05C70, 13F55}
\keywords{edge ideal, weighted oriented graph, dimension, depth, regularity, bipartite graph, tree}

\begin{document}
    \begin{abstract}
        We determine all possible triples of depth, dimension, and regularity of edge ideals of weighted oriented graphs with a fixed number of vertices. Also, we compute all the possible Betti table sizes of edge ideals of weighted oriented trees and bipartite~graphs with a fixed number of vertices. These results extend those of Kanno and Erey and Hibi from simple graphs to weighted oriented graphs.
    \end{abstract}
    
    \maketitle
    
    \section{Introduction}
         Let $\Bbbk$ be a field and $S = \Bbbk [x_1,x_2,\dots,x_n]$ denote a polynomial ring over $\Bbbk$. Further, let $G$ be a finite simple graph on the vertex set $\{x_i\colon\ i= 1,\dots, n\}$. The \emph{edge ideal} of $G$, denoted by $I(G)$, is the ideal of $S$ generated by $x_ix_j$ whenever $\{x_i,x_j\}$ forms an edge of $G$. Due to their connections to graphs and combinatorics, edge ideals are a central object of interest in commutative algebra.

        Studying the relationship among well-known homological invariants of edge ideals has recently been an active area of study in commutative algebra \cite{HMvT19, HKO19, hkkmt, HKKT21, KKS21, HKMvT22, hku,kanno2023tuples}. Such invariants include $\dim S/I(G)$, $\depth S/I(G)$, and $\reg S/I(G)$, to name a few; see \cref{sec:prelims} for definitions. Among these, there are clear and well-known relationships; e.g., we always have $1 \leq \depth S/I(G) \leq \dim S/I(G)$. In general, there are no more relationships between depth and dimension, even for edge ideals -- for any two integers $1 \leq a \leq b $, there exists a (star) graph $G$ such that $\depth S/I(G) = a$ and $\dim S/I(G) = b$.

        The problem becomes much more interesting when one restricts to \textbf{connected} graphs with a fixed number of vertices $n$. It was proven in \cite{hkkmt}\footnote{Theorem 2.8} that
        \[
            \ddUnw{n} \supseteq \left\{ (a,b) \in \mathbb{Z}^2\colon 1 \leq a \leq b \leq n - 1,\ a \leq b + 1 - \left\lceil \frac{b}{n - b} \right\rceil \right\},
        \]
        where
        \[ 
            \ddUnw{n} \coloneqq \left\{ (\depth S/I(G), \dim S/I(G))\colon \text{$G$ is connected on $n$ vertices} \right\}. 
        \]
        The authors in \cite{hku} conjectured that the converse also holds and proved it for $n\leq 12$. Kanno \cite{kanno2023tuples} settled the conjecture in its full generality. In brief, when we restrict to the class of connected graphs on $n$ vertices, an additional, more subtle constraint appears: \(a \leq b + 1 - \lceil b/(n - b) \rceil\). In \cite{kanno2023tuples}, Kanno also determined all the triples of depth, dimension, and (Castelnuovo-Mumford) regularity of edge ideals of connected graphs on $n$ vertices, which we denote by $\ddrUnw{n}$.

        In a different direction, edge ideals of weighted oriented graphs have also been an area of interest for study; see \cref{subsec:weighted-oriented} for definitions. As the name suggests, a \emph{weighted oriented graph} $(G,\w)$ is a graph whose vertices are assigned integers (called \emph{weights}, given by a \emph{weight function} $\w\colon V(G)\to \mathbb{N}$) and whose edges are oriented. The \emph{edge ideal} of a weighted oriented graph $(G,\w)$ on the vertex set $\{x_i \colon\ i = 1, \dots, n\}$ is defined as
        \[
            I(G, \w) \coloneqq (x_ix_j^{\w(x_j)} \colon \{x_i, x_j\} \text{ is an edge in $G$ oriented towards $x_j$}).
        \]
        Formally defined in \cite{PRT19}, this class of ideals quickly gained traction (\cite{GMSV18, HLMRV19, PRV21, casiday2021betti, CPR22, BCDMS23, CK24, Saha24}, to name a few examples), as it is a generalization of the classical edge ideals of graphs. Moreover, these ideals also appear naturally in coding theory as the initial ideals of the vanishing ideals of a finite set of points in a projective space. Among other applications, homological information of edge ideals of weighted oriented graphs has been used to derive some basic parameters of the Reed-Muller code of such finite sets of points (see, e.g., \cite{CNL17,MPV17}). Edge ideals of weighted oriented graphs, unlike those of graphs, are neither generated in the same degree, nor squarefree. Thus it is no surprise they exhibit different behaviors, e.g., their quotient rings can have depth 0. An important technical result of this article is \cref{thm:depth-zero-wo-graphs}, where we obtained a full characterization on when the quotient rings by edge ideals of weighted oriented graphs have depth $0$.
    
        In this paper, our goal is to extend the results of Kanno \cite{kanno2023tuples}, by determining $\ddUnw{n}$ and $\ddrUnw{n}$ in the context of connected weighted oriented graphs on $n$ vertices, where $n$ is fixed. We set up some notation:
        \begin{align*}
            \dd{n} &\coloneqq \!\begin{multlined}[t]
                \bigl\{ \bigl( \depth S/I(G), \dim S/I(G) \bigr) \colon \\
                \text{$G$ is a connected weighted oriented graph on $n$ vertices} \bigr\};
            \end{multlined}\\
            \ddr{n} &\coloneqq \!\begin{multlined}[t]
                \bigl\{ \bigl( \depth S/I(G), \dim S/I(G), \reg S/I(G) \bigr) \colon \\
                \text{$G$ is a connected weighted oriented graph on $n$ vertices} \bigr\}.
            \end{multlined}
        \end{align*}
        
        Our first main result is an explicit description of $\dd{n}$ and $\ddr{n}$ for $n\geq 2$. For the case $n = 1$, we note that the only ideal that can be defined from a (weighted oriented) graph on a single vertex is the zero ideal in $S = \Bbbk[x_1]$. 
        
        \begin{theorem}[{\cref{thm:dd-wo} and \cref{thm:ddr-wo}}]
            \label{thm:main-thm-1}
            For any $n \geq 2$, we have:
                \begin{enumerate}
                    \item \(\dd{n}  = \ddUnw{n} \cup \{ (0, b) \colon 1 \leq b \leq n - 2 \}\);
                    \item \(\ddr{n} = (\ddUnw{n} \times \NN) \cup \{ (0,b,c) \colon 1 \leq b \leq n - 2,\ c \geq 3\}\).
                \end{enumerate}
        \end{theorem}
        
        As another application of our techniques, we compute all the possible Betti table sizes of edge ideals of weighted oriented trees and (connected) bipartite graphs (see \cref{subsec:graph-theory} for graph-theoretic definitions). Given a homogeneous ideal $I$ of $S$, the Betti table of $S/I$ records all of its Betti numbers; its length is equal to the \emph{projective dimension} of $S/I$, denoted $\pd S/I$, and its width is equal to the \emph{regularity} of $S/I$, denoted $\reg S/I$. Erey and Hibi \cite{EH21} determined all possible Betti table sizes of edge ideals of trees and (connected) bipartite graphs, i.e., they determined the sets
        \begin{align*}
            \bptUnw{n} &\coloneqq \!\begin{multlined}[t]
                \{ (p,r)\in \mathbb{Z}^2\colon \text{there exists a bipartite graph $G$ on $n$ vertices} \\
                \text{with $\pd S/I(G) = p$ and $\reg S/I(G) = r$} \};
            \end{multlined}\\
            \treeUnw{n} &\coloneqq \!\begin{multlined}[t]
                \{ (p,r)\in \mathbb{Z}^2\colon \text{there exists a tree $G$ on $n$ vertices} \\
                \text{with $\pd S/I(G) = p$ and $\reg S/I(G) = r$} \}.
            \end{multlined}
        \intertext{We define an analog for weighted oriented graphs:}
            \bpt{n} &\coloneqq \!\begin{multlined}[t]
                \{ (p,r)\in \mathbb{Z}^2\colon \text{there exists a weighted oriented bipartite graph $G$}\\
                \text{on $n$ vertices with $\pd S/I(G)= p$ and $\reg S/I(G) = r$} \},
            \end{multlined}\\
            \tree{n} &\coloneqq \!\begin{multlined}[t]
                \{ (p,r)\in \mathbb{Z}^2\colon \text{there exists a weighted oriented tree $G$} \\
                \text{on $n$ vertices with $\pd S/I(G)= p$ and $\reg S/I(G) = r$} \}.
            \end{multlined}
        \end{align*}
    
        Our second main result is the following explicit description of $\bpt{n}$ and $\tree{n}$ for $n \geq 4$. 
        \begin{theorem}[{\cref{thm:bipartite-wo} and \cref{thm:tree-wo}}] \label{thm:main-thm-2}
            For $n\geq 4$, we have:
            \begin{enumerate}
                \item \(\bpt{n} = \{(p,r) \in \NN\colon \lceil \frac{n}{2} \rceil \leq p \leq n - 1\} \cup \{(n, r) \colon r \geq 4\}\);
                \item \(\tree{n} = \{\lceil \frac{n}{2} \rceil, \lceil \frac{n}{2} \rceil + 1, \dots, n - 1\} \times \NN\).
            \end{enumerate}
        \end{theorem}
        
        We note that it is straightforward to determine these sets when $n < 4$; for these values of $n$, there is exactly one connected bipartite graph and one tree on $n$ vertices. However, these cases require different formulae, and so we will not consider them.

        The paper is structured as follows. In \cref{sec:prelims}, we provide background information on monomial ideals, graph theory, and the tools needed in the remainder of the paper. In \cref{sec:wog-depth-0}, we provide a full characterization of weighted oriented graphs whose edge ideals have depth $0$. We then prove \cref{thm:main-thm-1}, which we divide into \cref{thm:dd-wo} and \cref{thm:ddr-wo}, in \cref{sec:triples}. Finally, \cref{sec:betti-tables} is dedicated to proving \cref{thm:main-thm-2}.

    \section*{Acknowledgements}
    
        The first author acknowledges supports by Infosys Foundation and the NSF grants DMS 2001368, 1801285, and 2101671. We would like to thank Professors Huy T\`{a}i H\`{a} and Adam Van Tuyl for the discussions regarding \cref{lem:regularity}. The second and third authors were supported by the Mathematics Department REU at the University of Utah and the NSF RTG Grant $\#1840190$. Finally, we all thank Dr. Tim Tribone for his mentorship throughout this project.

\section{Preliminaries}\label{sec:prelims}

    Let $I$ be a monomial ideal of the polynomial ring $S = \Bbbk[x_1, \dots, x_n]$. The \emph{height} of $I$, denoted $\height I$, can be computed using the formula  \[\height I = \min \{t: \text{ there exist } 1 \leq i_1 < \dots < i_t \leq n \text{ such that } I \subseteq (x_{i_1}, \dots, x_{i_t}) \}.\]
    The \emph{dimension} of the quotient ring $S / I$ is given by 
    \[
        \dim S/I = \dim S - \height I = n - \height I.
    \]
    The definitions of the other homological invariants we will consider, namely Betti numbers, projective dimension, regularity, and depth, rely on the notion of a \emph{minimal free resolution}. 

    \begin{definition}
        Let $\mc F = \cdots \xrightarrow{\partial} F_1 \xrightarrow{\partial} F_0 \to 0$ be a free resolution of the quotient ring $S/I$. We say that $\mc F$ is a \emph{minimal free resolution} if $\partial(F_i) \subseteq \mf m F_{i-1}$ for all $i$, where $\mf m = (x_1, \dots, x_n)$ is the irrelevant ideal. In other words, if we choose bases for the $F_i$ and represent the differential maps as matrices, all matrix entries lie in~$\mf m$. 
    \end{definition}

    It can be shown that minimal free resolutions exist for any module $S/I$ where $I$ is a monomial ideal, and are unique up to isomorphism. We define the \emph{projective dimension} of $S/I$, denoted by $\pd S/I$, to be the largest $i$ such that $F_i \neq 0$. We define the \emph{depth} of $S/I$ to be complementary to its projective dimension: \[\depth S/I = n-\pd S/I.\]
    We note that depth is traditionally defined differently; the above equality is called the Auslander-Buchsbaum formula.
    
    Now, for a monomial $m$, let $a(m)$ denote its exponent vector. Since $I$ is a monomial ideal, the module $S/I$, and thus its minimal free resolution $\mathcal{F}=(F_i)_{i\in \mathbb{Z}}$, is $\mathbb{N}^n$-graded, i.e., 
    \[
        F_i \cong \!\! \bigoplus_{m \in \mc M} S(-a(m))^{\beta_{i,a(m)}(S/I)},
    \]
    where $\mc M$ denotes the set of all monomials and $S(-a(m))$ denotes the free module $S$ shifted by $-a(m)$. The integers $\beta_{i,a(m)}(S/I)$ are called the \emph{multi-graded Betti numbers} of the module $S/I$.
    
    Finally, the \emph{regularity} of $S/I$ is defined to be
    \[
        \reg S/I = \max \{\deg m - i\colon\beta_{i,a(m)}(S/I) \neq 0\}.
    \]

    We recall a universal construction of a free resolution of $S/I$ for any monomial ideal $I$ with $\mingens(I) = \{m_1, \dots, m_q\}$. Let 
    \[
        \mathcal{T}_I \colon 0 \to T_q \to \cdots \to T_0 \to 0
    \]
    be the chain complex of $S$-modules, where 
    \[
        T_i = \!\!\!\! \bigoplus_{\substack{\sigma \subseteq \mingens(I)\\ \lvert \sigma \rvert = i}} \!\!\!\! Se_{\sigma}
    \]
    and the differentials are 
    \[
        \partial(e_\sigma) = \sum_{j = 1}^t (-1)^{j + 1} \frac{\lcm(\sigma)}{\lcm(\sigma\setminus \{m_{i_j}\})} e_{\sigma \setminus \{m_{i_j}\}},
    \]
    where $\sigma=\{m_{i_1}, \dots, m_{i_t}\}$ for some $1\leq i_1 < i_2 < \cdots < i_t \leq q$. Here, to make the complex $\mathbb{N}^n$-graded, we make the shift $Se_{\sigma} = S(-a(\lcm(\sigma)))$. Taylor \cite{Tay66} showed that the complex \(\mathcal{T}_I\) is a free resolution of \(S/I\); thus it is called the \textit{Taylor resolution} of \(S/I\).

    When the Taylor resolution of $S/I$ is minimal, its regularity can be computed by a simple formula. As we could not find a reference for such a formula, we prove it here.

    \begin{proposition}\label{prop:reg-Taylor}
        Let $I$ be a monomial ideal with $\mingens(I)=\{m_1, \dots, m_q\}$. If the Taylor resolution of $S/I$ is minimal, then 
        \[
            \pdim S/I = q
        \]
        and
        \[
            \reg S/I = \deg \left(\lcm(\mingens(I)) \right) - q.
        \]
    \end{proposition}

    \begin{proof}
        The projective dimension follows from the definition of the Taylor resolution. We will compute the regularity. By definition, we need to show
        \begin{equation*}
            \max \{\deg m - i \colon \beta_{i, a(m)}(S/I) \neq 0\} = \deg \left(\lcm(\mingens(I)) \right) - q.
        \end{equation*}
        We note that the last nonzero term of the Taylor resolution $\mathcal{T}_I$ is 
        \[
            T_q = S(-a(\lcm(\mingens(I)))),
        \]
        and, since this Taylor resolution is minimal by hypothesis, 
        \[
            \beta_{q, a(\lcm(\mingens(I)))}(S/I) = 1 \neq 0.
        \]
        Thus, \(\deg \left( \lcm(\mingens(I)) \right) - q \in  \{\deg m - i \colon \beta_{i, a(m)}(S/I) \neq 0\}\).

        It now suffices to show that if $\beta_{i, a(m)}(S/I)\neq 0$ for some monomial $m$ and integer $i$, we have 
        \begin{equation}\label{reg-Taylor}
            \deg m - i \leq \deg \left( \lcm(\mingens(I)) \right) - q. 
        \end{equation}
        Indeed, since the Taylor resolution $\mathcal{T}_I$ is minimal, there exists $\sigma \subseteq \mingens(I)$ such that $m = \lcm(\sigma)$ and $i = \lvert \sigma \rvert \leq \lvert \mingens(I) \rvert = q$. Without loss of generality, assume that $\sigma = \{m_1, \dots, m_i\}$. If $i = q$, then \cref{reg-Taylor} is vacuously true. Now assume that $i < q$. For any $s \in [q - 1]$, the coefficient $\lcm(m_1, \dots, m_{s}, m_{s + 1})/$ $\lcm(m_1, \dots, m_s)$ appears in the Taylor resolution $\mathcal{T}_I$. Since $\mathcal{T}_I$ is minimal, the quotient $\lcm(m_1, \dots, m_{s}, m_{s + 1})/$ $\lcm(m_1, \dots, m_s)$ is not 1, and, in particular, we have 
        \[
            \deg \left( \lcm(m_1, \dots, m_{s}, m_{s + 1}) \right) \geq \deg \left( \lcm(m_1,\dots, m_{s})\right) + 1.
        \]
        Applying this inequality repeatedly, we obtain
        \begin{align*}
            \deg \left( \lcm(\mingens(I)) \right) - q &= \deg \left( \lcm(m_1,\dots, m_{q - 1}, m_{q}) \right) - q\\
            &\geq \deg \left(\lcm(m_1, \dots, m_{q - 1}) \right) - (q - 1)\\
            &\vdotswithin{\geq}\\
            &\geq \deg \left(\lcm(m_1, \dots, m_{i}) \right) - i\\
            &= \deg \left(\lcm(\sigma) \right) - i, 
        \end{align*}
        as desired.
    \end{proof} 

Next we present a result that helps us manipulate regularity of monomial ideals.

\begin{lemma}\label{lem:regularity}
    Let $I$ be a squarefree monomial ideal and $r\geq 0$ an integer. Then there exists a variable $x$ such that after replacing $x$ with $x^{r+1}$ for some integer $n$, we obtain a monomial ideal $J$ (in the same polynomial ring) such that
    \[\pdim S/J = \pdim S/I\]
    and 
    \[
    \reg S/J = \reg S/I + r.
    \]
\end{lemma}

\begin{proof}
    Let $S=\Bbbk[x_1,\dots, x_n]$ be the polynomial ring that contains $I$. Set $S'=S[y]$ where $y$ is a new variable. Fix an index $i$. Let $I'$ be a monomial ideal of $S'$ obtained from $I$ by replacing $x_i$ with $y$, and $J_i$ be the ideal obtained from $I$ by replacing $x_i$ with $x_i^{r+1}$.

    Let $\mathcal{F}$ be the minimal resolution of $S'/I'$ over $S'$. By considering a lex order where $x_i$ is the largest, it is clear that $x_i-y$ is a regular element of both $S'/I'$ and $S'$. Hence by
    \cite[Proposition 1.1.5]{BH98}, $\mathcal{F} \otimes_{S'} (S'/(x_i-y))$ is the minimal resolution of $S/I$ over $S$. Similarly, $x_i^{r+1}-y$ is also a regular element of both $S'/I'$ and $S'$. Hence by
    \cite[Proposition 1.1.5]{BH98} again, $\mathcal{F} \otimes_{S'} S'/(x_i^{r+1}-y)$ is the minimal resolution of $S/J_i$ over $S$. In particular, this means that $\pdim S/J_i = \pdim S/I$, and that the minimal free resolution of $S/I$ over $S$, after replacing $x_i$ with $x_i^{r+1}$ in all matrices, becomes that of $S/J_i$ over $S$. This gives us a direct relation between the multigraded Betti numbers of the two ideals:
    \[
    \beta_{j,m}(S/I) = \begin{cases}
        \beta_{j,mx_i^r}(S/J_i) & \text{ if } x_i\mid m,\\
        \beta_{j,m}(S/J_i) & \text{ otherwise},
    \end{cases}
    \]
    for any monomial $m$ and index $j$.
    
    Now we will show the regularity equality for a suitable choice of $i$. Set $\reg S/I = t$, i.e., there exists a monomial $m_0$ and an index $j_0$ such that
    
    \begin{enumerate}
        \item \(\beta_{j_0,m_0}(S/I)\neq 0\),
        \item \(\deg(m_0)-j_0 = t\),
        \item If \(\beta_{j,m}(S/I)\neq 0\), then \(\deg(m)-j \leq t\).
    \end{enumerate}

    Let $i$ be an index such that $x_i$ divides $m_0$. Without loss of generality, assume that $x_1$ divides $m_0$. We have the formulas for the Betti numbers of $J_1$:
    \[
    \beta_{j,m}(S/J_1) = \begin{cases}
        \beta_{j,m/x_1^r}(S/I) & \text{ if } x_1\mid m,\\
        \beta_{j,m}(S/I) & \text{ otherwise},
    \end{cases}
    \]
    for any monomial $m$ and index $j$. It is clear that $\beta_{j_0,m_0x_1^r}(S/J) \neq 0$, and 
    \[
    \deg(m_0x_1^r) -j_0 = \deg(m_0) + r - j_0 = t+r.
    \]
    Finally, assume that $\beta_{j,m}(S/J_1)\neq 0$ for some index $j$ and monomial $m$. It suffices to show that $\deg(m)-j \leq t+r$. Indeed, if $x_1\nmid m$, then by the Betti formulas, we have $\beta_{j,m}(S/I)\neq 0$, which implies that $\deg(m)-j\leq t\leq t+r$, as desired. On the other hand, if $x_1\mid m$, then $\beta_{j,m/x_i^r}(S/I)\neq 0$, which implies that 
    \[
    \deg(m)-j = (\deg(m/x_i^r) + r )-j = (\deg(m/x_i^r) -j )+r \leq t+r,   \]
    as desired.
\end{proof}

\subsection{Graph theory terminology}\label{subsec:graph-theory}

In this section we recall some special types of graphs.

A \emph{graph} $G$ consists of a \emph{vertex set} $V(G)$ and an \emph{edge set} $E(G)$, whose elements are sets of the form \(\{x_i, x_j\}\) for distinct \(x_i, x_j \in V(G)\). If \(\{x_i, x_j\} \in E(G)\), then \(x_i\) and \(x_j\) are connected by an edge in $G$. As a consequence of this definition, all graphs considered in this paper will be simple, i.e., no loops or repeated edges are allowed.

An \emph{oriented graph} is a graph $G$ in which each edge has an \emph{orientation}, i.e., a direction. A \emph{weighted oriented graph} is an oriented graph $G$ equipped with a weight function $\w$ which assigns to each vertex of $G$ a positive, integer-valued \emph{weight}. We will denote weighted oriented graphs by ordered pairs of the form $(G, \w)$. 

Given two graphs $G$ and $G'$, we say $G'$ is a \emph{subgraph} of $G$ if \(V(G') \subseteq V(G)\) and $E(G') \subseteq E(G)$. If any edge of $G$ connecting two vertices of $G'$ is an edge of $G'$, we further say $G'$ is an \emph{induced subgraph}. Given weighted oriented graphs $(G, \w)$ and $(G', \w')$, $(G', \w')$ a subgraph of $(G, \w)$ if
\begin{enumerate}
    \item $G'$ is a subgraph of $G$,
    \item the orientation of any edge in $G'$ agrees with its orientation in $G$, and
    \item $\w'$ is the restriction of $\w$ to $V(G')$.
\end{enumerate}
Finally, $(G', \w')$ is an induced subgraph of $(G, \w)$ if $G'$ is an induced subgraph of $G$. For convenience of notation, we will denote subgraphs of $(G,\w)$ by $(G',\w)$.

Let $C_n$ denote the cycle graph on $n$ vertices, where $n\geq 3$. A connected graph is called a \emph{tree} if it has no induced cycle, a \emph{unicyclic graph} if it has exactly one induced cycle, and a \emph{bipartite graph} if it does not have any induced odd cycle. In particular, trees are bipartite.

    \begin{figure}[ht!]
        \centering
        \begin{subfigure}{0.49\linewidth}
            \centering
            \includegraphics[width = 0.6\linewidth]{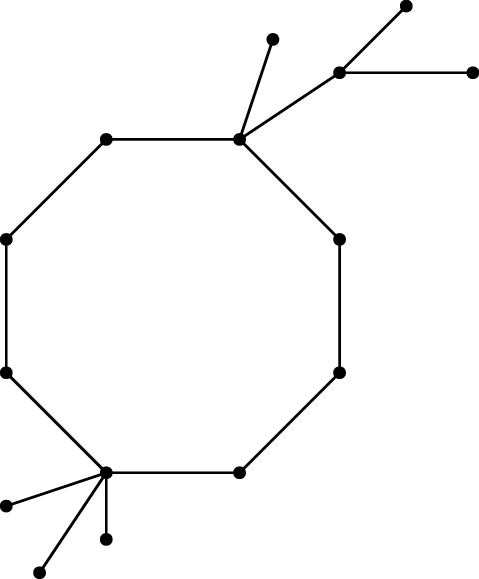}
        \end{subfigure}
        \begin{subfigure}{0.49\linewidth}
            \centering
            \includegraphics[width = 0.29\linewidth]{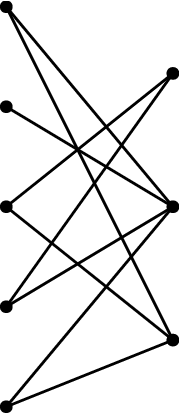}
        \end{subfigure}
        \caption{Examples of a unicyclic graph (left) and a bipartite graph (right).}
        \label{fig:ex-unicyclic-bipartite}
    \end{figure}

We call a graph a \emph{maximal pseudo-forest} if each of its connected components is a unicyclic graph.

Any tree $G$ can be considered \emph{rooted} at any of its vertices, i.e., vertices of $G$ can be labeled $\{x, x_{1,1},\dots, x_{1,n_1},\dots, x_{l,1},\dots, x_{l, n_l}\}$ for some integers $l, n_1,\dots, n_l$  such that the (unique) shortest path between $x_{i,j}$ and $x$ is of length $i$ for any $i$ and $j$. This is illustrated in \cref{fig:rooted-tree}.

    \begin{figure}[ht!]
        \centering
        \begin{subfigure}{0.49\linewidth}
            \centering
            \includegraphics[width = 0.49\linewidth]{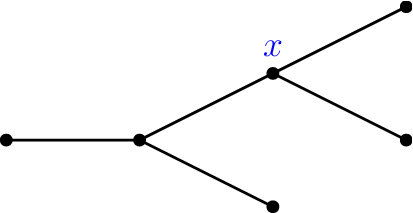}
        \end{subfigure}
        \begin{subfigure}{0.49\linewidth}
            \centering
            \includegraphics[width = 0.49\linewidth]{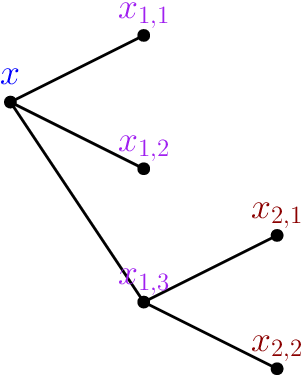}
        \end{subfigure}
        \caption{A tree (left) and the same tree as rooted at vertex \(x\) (right).}
        \label{fig:rooted-tree}
    \end{figure}

Recall that any unicyclic graph can be considered to be a cycle graph whose vertices are glued to trees, see e.g., \cite[Proof of Lemma 4]{unicyclic}. For an example, see \cref{fig:unicyclic-decomposition}.

    \begin{figure}[ht!]
        \centering
        \begin{subfigure}{0.49\linewidth}
            \centering
            \includegraphics[width = 0.49\linewidth]{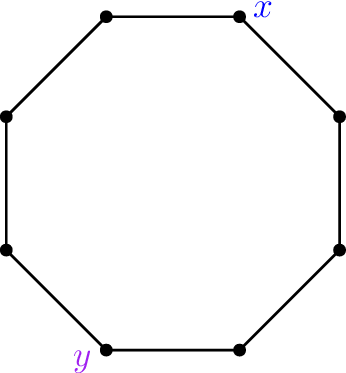}
        \end{subfigure}
        \begin{subfigure}{0.49\linewidth}
            \centering
            \includegraphics[width = 0.49\linewidth]{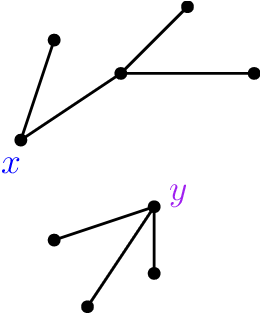}
        \end{subfigure}
        \caption{The decomposition of the unicyclic graph in \cref{fig:ex-unicyclic-bipartite} into a cycle and two trees.}
        \label{fig:unicyclic-decomposition}
    \end{figure}

We call an oriented unicyclic graph \emph{naturally oriented} if the edges of its (unique) induced cycle $C_n$ are oriented in one direction, and the edges of the trees glued to vertices of $C_n$ are oriented away from them. In other words, if we consider the trees to be rooted at vertices of $C_n$, then their edges are oriented away from the root. We illustrate these concepts in the following pictures.

    \begin{figure}[ht!]
        \centering
        \begin{subfigure}{0.51\linewidth}
            \centering
            \includegraphics[width = 0.51\linewidth]{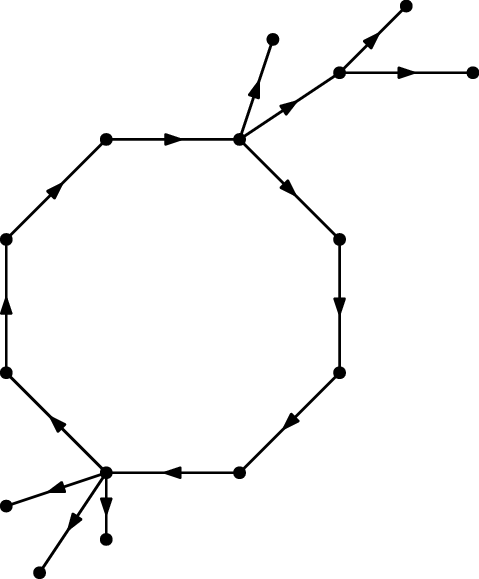}
        \end{subfigure}
        \begin{subfigure}{0.49\linewidth}
            \centering
            \includegraphics[width = 0.49\linewidth]{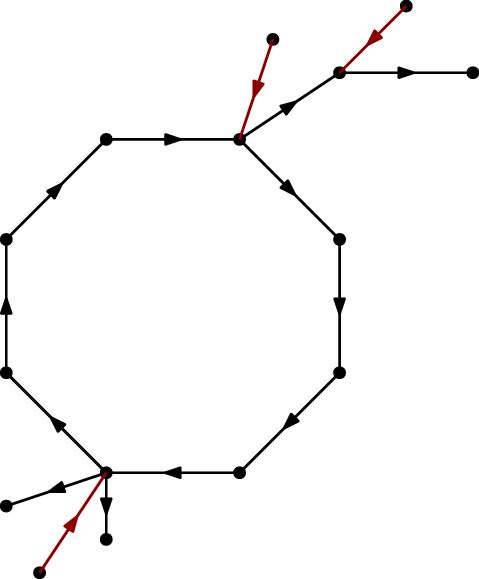}
        \end{subfigure}
        \begin{subfigure}{0.49\linewidth}
            \centering
            \includegraphics[width = 0.49\linewidth]{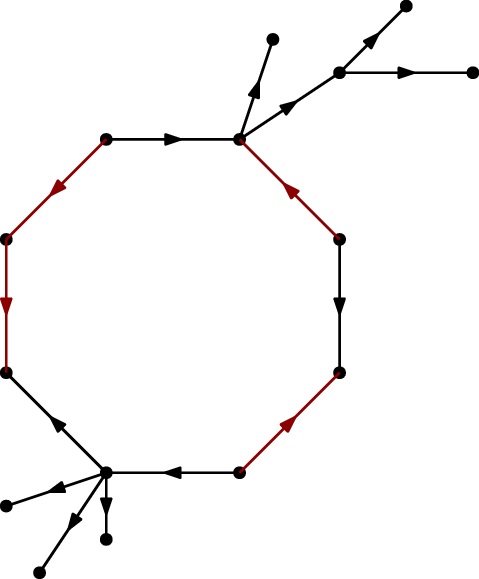}
        \end{subfigure}
        \caption{A naturally oriented unicyclic graph (above) and two ``unnatural'' orientations of the same graph (below).}
        \label{fig:unicyclic-orientation}
    \end{figure}

We call an oriented maximal pseudo-forest \emph{naturally oriented} if each of its connected components is a naturally oriented unicyclic graph.

\subsection{Edge ideals of weighted oriented graphs}\label{subsec:weighted-oriented}

For convenience, we recall some definitions from the introduction here. If $G$ is a graph on the vertices $x_1, \dots, x_n$, then $G$ defines a monomial ideal of the polynomial ring $S = \Bbbk[x_1, \dots, x_n]$ generated by the set
\[\{x_i x_j \colon \text{ the vertices } x_i \text{ and } x_j \text{ are joined by an edge in } G\}\]
We call this ideal the \emph{edge ideal} of $G$, and denote it by $I(G)$. A larger class of monomial ideals may be described using weighted oriented graphs, which are graphs with additional structure. 

Given a weighted oriented graph $(G, \w)$, we define the \emph{edge ideal} $I(G, \w)$ \emph{of} $(G, \w)$ to be the monomial ideal generated by 
\[\{x_i x_j^{\w(x_j)} \colon \text{ there is an edge oriented from } x_i \text{ to } x_j \text{ in } G\}.\]

\begin{example}
    \begin{figure}[ht!]
        \centering
        \begin{subfigure}{0.49\linewidth}
            \centering
            \includegraphics[width = 0.49\linewidth]{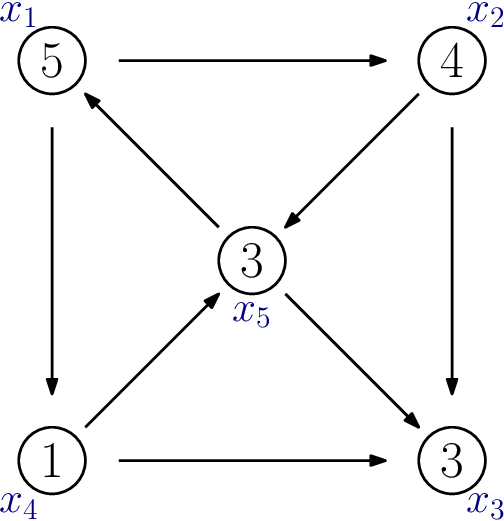}
        \end{subfigure}
        \begin{subfigure}{0.49\linewidth}
            \centering
            \includegraphics[width = 0.49\linewidth]{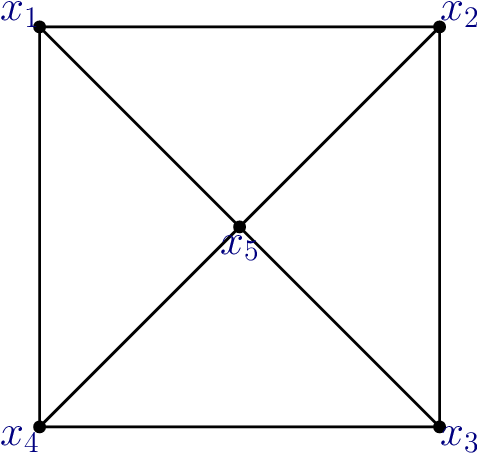}
        \end{subfigure}
        \caption{A weighted oriented graph (left) and its underlying base graph (right).}
        \label{fig:ex-wog}
    \end{figure}
    Let \((G,\w)\) refer to the weighted oriented graph in \cref{fig:ex-wog}. The edge ideal of \((G,\w) \) is 
    \[I(G, \w) = (x_1 x_2^4, x_1 x_4, x_2 x_3^3, x_2 x_5^3, x_4 x_3^3, x_4 x_5^3, x_5 x_1^5, x_5 x_3^3).\] 
    Notice that if all of the weights were 1, this edge ideal would be \((x_1 x_2, x_1 x_4, x_2 x_3,\) \(x_2 x_5, x_4 x_3, x_4 x_5, x_5 x_1, x_5 x_3)\), which is the same as the edge ideal of the base graph.
\end{example}

To simplify notation, we use $\depth(G, \w), \dim(G, \w), \reg(G, \w),$ and $\pdim(G, \w)$ to denote $\depth S/I(G, \w)$, $\dim S/I(G, \w)$, $\reg S/I(G, \w)$, and $\pdim S/I(G, \w)$, respectively. When we refer to the edge ideal $I(G)$, or equivalently, the edge ideal of the weighted oriented graph $I(G,\w)$ where $\w(v)=1$ for each $v\in V(G)$, we will simply use  $\depth G, \dim G, \reg G,$ and $ \pdim G$.

With this terminology, we now give a quick study on the relationship between these homological invariants of weighted oriented graphs and those of their base graphs.

    Our first observation is that the dimension remains unchanged. In other words, the weights on the graph $G$ do not affect the height of the corresponding edge ideal.

    \begin{lemma}[\protect{{\cite[Proposition 1.2.4]{HH11-book}}}]\label{lem:dim(G w)=dim(G)}
        If $(G,\w)$ is a weighted oriented graph, then $\sqrt{I(G,\w)} = I(G)$. In particular, $\dim (G,\w) = \dim G$.
    \end{lemma}

On the other hand, in general, we only have an inequality when it comes to depth and regularity.

    \begin{proposition}\label{prop:depth(G w)-leq-depth(G)}
        Let $(G,\w)$ be a weighted oriented graph. Then $\depth (G,\w) \leq \depth G$. 
    \end{proposition}

    \begin{proof}
        This follows from \cite[Observation 37]{casiday2021betti}.
    \end{proof}

We now give an application of \cref{lem:regularity}, which allows us to construct a weighted oriented graph with controlled homological invariants from a given base graph.

\begin{lemma}\label{lem:construct-new-graph}
    Let $G$ be a graph and $r\geq 0$ an integer. Then there exists a weight function $\w$ on the vertices of $G$ and an orientation on $G$ such that 
    \[
    \dim(G,\w) = \dim G, \ \depth(G,\w) = \depth G \text{ and } \reg(G,\w) = \reg G + r.
    \]
\end{lemma}
\begin{proof}
    By \cref{lem:regularity}, there exists a vertex $x$ of $G$ such that replacing $x$ with $x^{r+1}$ in the edge ideal $I(G)$ results in a monomial ideal $J$ such that $\depth S/J = \depth G$ and $\reg S/J = \reg G + r$. By definition, $J$ equals the ideal $I(G,\w)$ where
    \begin{enumerate}
        \item $G$ is considered an oriented graph with all edges involving $x$ oriented towards it, and other edges are oriented randomly; and
        \item $\w(x)=r+1$ and $\w(v)=1$ for any $v\in V(G)\setminus \{x\}$.
    \end{enumerate}
    The results on depth and regularity then follow. The result on dimension follows from \cref{lem:dim(G w)=dim(G)}. This concludes the proof.
\end{proof}

\subsection{Betti splittings}

A popular method to compute the homological invariants of a monomial ideal $I$, e.g., Betti numbers, projective dimension, or regularity, is to find a decomposition $I=J+K$ where $J,K,$ and $J+K$ are such that $\mingens(I)=\mingens(J) \bigsqcup \mingens(K)$ and 
\[
\beta_{i,a(m)}(S/I)= \beta_{i,a(m)}(S/J)+\beta_{i,a(m)}(S/K)+\beta_{i-1,a(m)}(S/J\cap K)
\]
for any $i\geq 1$ and monomial $m$. Such a decomposition is called a \emph{Betti splitting}. There are methods to create one for ideals of special forms. In the context of this paper, we will only use the following result. 

\begin{theorem}[\protect{{\cite[Corollary 2.7]{FHV09}}}] \label{thm:betti-splittings}
    Let $x$ be a variable. Let $I=J+K$ be monomial ideals in $S$ such that $\mingens(J)=\{ m \in \mingens(I)\colon x\mid m \}$ and $\mingens(K) = \mingens(I)\setminus \mingens(J)$. Assume that $J, K\neq 0$. Then if $K$ has a linear minimal resolution, then $I=J+K$ is a Betti splitting. In particular, we have
    \begin{align*}
        \pd S/I &= \max \{ \pd(S/J), \pd(S/K), \pd(S/J\cap K) +1 \}\\
        \intertext{and}
        \reg S/I &= \max \{ \reg S/J, \reg S/K, \reg S/J\cap K -1 \}.
    \end{align*}
\end{theorem}

\section{Weighted oriented graphs with depth zero}\label{sec:wog-depth-0}

In this section we study weighted oriented graphs $(G,\w)$ with depth $0$. By definition, it means that they have maximum projective dimension. Alesandroni \cite{Ale20} has fully characterized monomial ideals with maximum projective dimension in terms of dominant monomials. We first recall some terminology from \cite{Ale20}. Let $M$ be a set of monomials and $m\in M$. Then $m$ is called \emph{dominant} in $M$ if there exists a variable $x$ and an integer $k$ such that $m$ is the only monomial in $M$ that is divisible by $x^k$. Equivalently, $m$ is dominant in $M$ if $\lcm(M)\neq \lcm(M\setminus \{m\})$. Moreover, we call the set $M$ \emph{dominant} if each monomial in $M$ is dominant in $M$. 

For two monomials $m$ and $m'$, we say $m$ \emph{strongly divides} $m'$ if $m$ divides $m'$ and whenever $x$ is a variable that divides $m$, it divides $m'/m$. We recall the main theorem in \cite{Ale20}.

\begin{theorem}[\protect{{\cite[Corollary 5.3]{Ale20}}}]\label{thm:dominant-Ales}
    Let $I$ be a monomial ideal in the polynomial ring $S=\Bbbk[x_1,\dots, x_n]$, where $\Bbbk$ is a field. Then $\depth S/I=0$ if and only if $\mingens(I)$ contains a dominant set $M$ of cardinality $n$, and no monomial in $\mingens(I)$ strongly divides $\lcm(M)$.
\end{theorem}

The goal of this section is to translate this result to the context of edge ideals of weighted oriented graphs. We first go through some examples to illustrate the~concepts.

\begin{example}
    In the set of monomials $M \coloneqq \{xy^3,y^3z,x^4z\}$, the monomial $x^4z$ is dominant since $\lcm(M) = x^4y^3z \neq xy^3z=\lcm(M\setminus \{x^4z\})$, while the other two are not dominant following similar arguments.
\end{example}

We can fully characterize the dominant generators in $\mingens(I(G,\w))$ for a weighted oriented graph $(G,\w)$. Recall that a vertex of a graph is called a \emph{leaf} if it is connected to exactly one other vertex.

\begin{proposition}\label{prop:dominant-generator}
    Let $(G,\w)$ be a weighted oriented graph. A generator $uv^{\w(v)}$ of $I(G,\w)$ is dominant in $\mingens(I(G,\w))$ if and only if either of the following holds:
    \begin{enumerate}
        \item Either $u$ or $v$ is a leaf of $G$.
        \item $uv$ is the only edge of $(G,\w)$ that is oriented towards $v$, and $\w(v)>1$.
    \end{enumerate}
\end{proposition}
\begin{proof}
    Assume either of the two conditions holds. If $u$ (resp, $v$) is a leaf, then the only monomial in $\mingens(I(G,\w))$ that is divisible by $u$ (resp, $v$) is $uv^{\w(v)}$. Hence $uv^{\w(v)}$ is dominant in $\mingens(I(G,\w))$, as desired. Now assume that $uv$ is the only edge of $(G,\w)$ that is oriented towards $v$, and $\w(v)>1$. By definition, $uv^{\w(v)}$ is the only monomial in $\mingens(I(G,\w))$ divisible by $v^2$, and thus is dominant in $\mingens(I(G,\w))$, as desired.

    Now assume that both conditions fail. Condition $(1)$ failing means that neither $u$ nor $v$ is a leaf of $G$. In other words, there exist vertices $a,b$ of $G$, distinct from $u$ and $v$, such that $au$ and $bv$ are edges of $G$. If $\w(v)= 1$, then $uv^{\w(v)}=uv$ cannot be dominant in $\mingens(I(G,\w))$ due to the existence of the edges $au$ and $bv$ in $G$. Thus we assume that $\w(v)\geq 2$. Since condition $(2)$ fails, there exists another edge of $(G,\w)$ that is oriented towards $v$, i.e., $\mingens(I(G,\w))$ contains $cv^{\w(v)}$ for some vertex $c$. Coupling this with the existence of the edge $au$ in $G$, we conclude that $uv^{\w(v)}$ is not dominant in $\mingens(I(G,\w))$.
\end{proof}

Next we characterize the weighted oriented graphs $(G,\w)$ where $\mingens(G,\w)$ is a dominant set of cardinality $|V(G)|$.

\begin{proposition}\label{prop:dominant-unicyclic}
    Let $(G,\w)$ be a weighted oriented graph on $n$ vertices. Then $\mingens(G,\w)$ is a dominant set of cardinality $n$ if and only if $G$ is a weighted naturally oriented maximal pseudo-forest where any vertex $v$ of $G$ is either a leaf or satisfies  $\w(v)\geq 2$.
\end{proposition}
\begin{proof}
    If $H_1,\dots, H_k$ are the connected components of $G$, then $\mingens(G,\w)$ is the disjoint union of $\{\mingens(H_i,\w)\}_{i=1}^k$, and these sets are in different variables. It is then straightforward that a monomial in $\mingens(I(G,\w))$ is dominant in $\mingens(I(G,\w))$  if and only if it is dominant in $\mingens(I(H_i,\w))$ where $H_i$ contains the corresponding edge. Therefore we can assume that $G$ is connected. The backwards direction follows immediately from \cref{prop:dominant-generator}.

    We now prove the forward implication. Assume that $\mingens(G,\w)$ is a dominant set of cardinality $n$. In particular, this means that $G$ has $n$ edges, and since $G$ is connected, $G$ is a unicyclic graph. Without loss of generality, assume that $\{x_1x_2,\dots, x_{k-1}x_k, x_kx_1\}$ is the cycle in $G$, and that $x_1x_2$ is oriented towards $x_2$. Since $x_2$ is not a leaf of $G$, by \cref{prop:dominant-generator}, $\w(x_2)\geq 2$, and $x_2x_3$ is oriented towards $x_3$. Repeating this argument, the cycle in $G$ is oriented in one direction, and $\w(x_i)\geq 2$ for any $i \in \{1, \dots, k\}$. 
    
    If $G$ is exactly this cyclic subgraph, then we obtain what we desired. So we assume that there are trees attached to the cycle. Let $T$ be the induced tree of $G$ that is attached to $x_1$. We consider $T$ to be a tree rooted at $x_1$, i.e., we can assume the vertices of this tree are 
    \[
    \{x_{ij}\colon 0\leq i\leq t, 1\leq j \leq k_i \}
    \]
    where $t, k_1,\dots, k_t$ are integers, and $dist_T(x_1,x_{ij}) = i$. By definition, we must have $k_0=1$ and $x_{01}=x_1$. Since $x_kx_1$ is already oriented towards $x_1$ in $G$, by \cref{prop:dominant-generator}, for any $1\leq j\leq k_1$, the edges $x_1x_{1j}$ must be oriented towards $x_{1j}$ in $G$, and either $x_{1j}$ is a leaf or $\w(x_{1j})\geq 2$. Repeating this argument, the tree $T$ is oriented away from $x_1$, and all of its vertices are either leaves, or are of weight at least 2. The same arguments apply to the other trees attached to the cycle in $G$. In conclusion, $(G,\w)$ is a weighted naturally oriented maximal pseudo-forest whose vertices are either leaves or are of weight at least 2, as desired. 
\end{proof}

Now we are ready to prove the main result of this section, which allows us to easily identify when a weighted oriented graph has depth $0$. Spoiler alert: the other condition in \cref{thm:dominant-Ales} is automatically satisfied.

\begin{theorem} \label{thm:depth-zero-wo-graphs}
    A weighted oriented graph has depth $0$ if and only if it contains, as a subgraph with the same number of vertices, a weighted naturally oriented maximal pseudo-forest whose vertices are either leaves or are of weight at least 2.
\end{theorem}

\begin{proof}
    Using \cref{thm:dominant-Ales} and \cref{prop:dominant-unicyclic}, it suffices to show that if $(G,\w)$ contains, as a subgraph with $|V(G)|$ vertices,  a weighted naturally oriented maximal pseudo-forest $(H,\w)$, then no monomial in $\mingens(I(G,\w))$ strongly divides $\lcm(\mingens(I(H,\w)))$. Indeed, with the above hypothesis, for every vertex of $G$, there exists an edge of $G$ that is oriented towards it. Thus we have 
    \[
    \lcm(\mingens(I(H,\w))) = \prod_{x \in V(G)} x^{\w(x)}.
    \]
    Since any monomial in  $\mingens(I(G,\w))$ is of the form $uv^{\w(v)}$, it does not strongly divides $\lcm(\mingens(I(H,\w)))$ by definition, as desired.
\end{proof}

It is worth noting that a monomial ideal $I$ has a minimal Taylor resolution if and only if $\mingens(I)$ is a dominant set \cite[Theorem 4.4]{Ale20}. Thus, the following result follows directly from \cref{prop:dominant-unicyclic} and \cref{prop:reg-Taylor}.

\begin{corollary}\label{cor:regularity-pseudo-forest}
    Let $(G,\w)$ be a weighted naturally oriented maximal pseudo-forest where each vertex is either a leaf or has weight at least 2. Then 
    \[
    \pdim(G,\w) = |E(G)|
    \]
    and
    \[
    \reg(G,\w)= \sum_{v\in V(G)} \w(v) - |E(G)|.
    \]\qed
\end{corollary}

\section{Depth, dimension, and regularity triples of weighted oriented graphs}\label{sec:triples}
    In this section, we will completely describe the sets $\dd{n}$ and $\ddr{n}$ for any $n\geq 2$.  Our motivation is the following theorem by Kanno.

\begin{theorem}[\protect{{\cite[Corollary 1.7]{kanno2023tuples}}}]\label{thm:Kanno}
    For any $n\geq 2$, we have
    \[
    \ddUnw{n} = \left\{ (a,b) \in \mathbb{Z}^2\colon 1\leq a\leq b \leq n-1,  a\leq b+1 - \left\lceil \frac{b}{n-b} \right\rceil \right\}.
    \]
\end{theorem}

We first compute some homological invariants of edge ideals of star graphs and complete graphs (and their weighted oriented variants). Recall that a \emph{star graph} is one where all edges contain a fixed vertex, while a \emph{complete graph} is one whose any two vertices form an edge.

\begin{lemma}\label{thm:dim_complete_graph}
        Let $K_n$ be a complete graph on $n \ge 2$ vertices $x_1,x_2,\dots,x_n$. Then $\dim (K_n,\w) = 1$ for any weight function $\w$.
    \end{lemma}

    \begin{proof}
        In light of \cref{lem:dim(G w)=dim(G)}, it suffices to show that $\dim K_n = 1$ for any $n \ge 3$. The edge ideal $I \coloneqq I(K_n)$ is generated by all pairs $x_ix_j$ with $i < j$. First notice that $I \subseteq (x_1,x_2,\dots,x_{n-1})$ implying that $\height(I) \leq n-1$. 
        
        On the other hand, suppose that $I \subseteq (x_{i_1},x_{i_2},\dots,x_{i_t})$ for $t < n-1$. Since $t < n-1$, there exists $x_r,x_s$ such that $x_r,x_s \not\in \{x_{i_1},x_{i_2},\dots,x_{i_t}\}$. This implies $x_rx_s \not\in (x_{i_1},x_{i_2},\dots,x_{i_t})$ contradicting that $I \subseteq (x_{i_1},x_{i_2},\dots,x_{i_t})$. Hence, $\height(I) \ge n-1$ which combined with the above gives $\height(I) = n-1$ and $\dim K_n = 1$
    \end{proof}

It is known that the dimension of a graph is at most $n-1$ (cf.~\cite[Theorem~2.9]{hkkmt}), and it reaches $n-1$ if and only if the graph is a star (cf.~\cite[Lemma~1.3]{hkkmt}). In view of \cref{lem:dim(G w)=dim(G)}, it is of no surprise that the same holds for weighted oriented graphs, despite the fact that the dimension of those can reach $n$. We provide a proof for completion.

\begin{lemma} \label{lem:star}
    Let $G$ be a graph on $n$ vertices $V(G)=\{x_1,\dots, x_n\}$. Then \\$\dim(G, \w) = n-1$ if and only if $G$ is a star. 
\end{lemma}

\begin{proof}
    We have
    \begin{align*}
        \dim (G, \w) = n - 1 &\iff \height (G, \w) = 1\\
        &\iff I(G, \w) \subseteq (x_i) \text{ for some } 1 \leq i \leq n\\
        &\iff \text{all edges in } G \text{ meet } x_i \text{ for some } i\\
        &\iff G \text{ is a star}.\qedhere
    \end{align*}
\end{proof}

\cref{lem:regularity} allows us to modify base graphs into weighted oriented graphs with manipulatable regularity, but the projective dimension would remain unchanged. As the depth of the base graph is at least 1, we virtually do not know much about those of depth 0. Next we will construct a family of weighted oriented graphs with depth $0$ that will be helpful in our proofs.

    \begin{chunk} \label{special-wographs}
        Let $S = \Bbbk [x_1,x_2,x_3,y_1,y_2,\dots,y_t,z_1,\dots, z_l]$ where $t,l \ge 0$. Let $(\mc G_{t,l,r},\w_{t,l,r})$ be the complete graph on the vertices $\{x_1,x_2,x_3,y_1,\dots,y_t\}$ with leaves $z_1,\dots, z_l$ attached to $x_1$ with the following weight function and orientation:
            \begin{itemize}
                \item $\w_{t,l,r}(x_1)=2+r$.
                \item $\w_{t,l,r}(x_2)=\w_{t,l,r}(x_3) = 2$.
                \item For $j=1,\dots,t$, $\w_{t,l,r}(y_j) = 1$.
                \item For $j=1,\dots,r$, $\w_{t,l,r}(z_j)=1$.
                \item  The $3$-cycle with vertices $x_1, x_2, x_3$ is naturally oriented.
                \item Any edge of the form $x_iy_j$ is oriented toward $y_j$.
                \item Any edge of the form $y_iy_j$ is arbitrarily oriented.
                \item Any edge of the form $x_1z_j$ is oriented toward $z_j$.
            \end{itemize}

        Explicitly, the edge ideal of $(\mc G_{t,l,r}, \w_{t,l,r})$ is
        \begin{equation*}
            \begin{split}
                I(\mc G_{t,l,r}, \w_{t,l,r}) = (x_1^{2+r} x_2, x_2^2x_3, x_3^2x_1) &+ \sum_{j=1}^t(x_1y_j,x_2y_j,x_3y_j)\\
                &+ (\{y_iy_j \colon i<j\} + \{x_1 z_j\colon 1 \leq j \leq l\}).
            \end{split}
        \end{equation*}

        \begin{figure}[ht!]
            \centering
            \begin{subfigure}{0.49\linewidth}
                \centering
                \includegraphics[width = 0.49\linewidth]{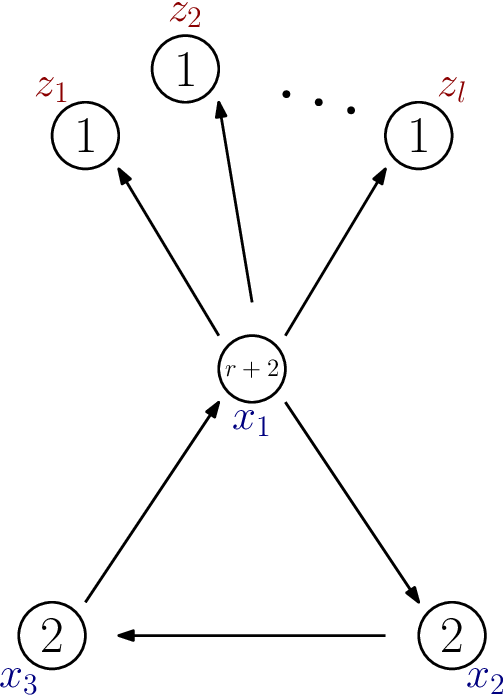}
                \subcaption{\((\mc G_{0,l,r}, \w_{0,l,r})\).}
            \end{subfigure}
            \begin{subfigure}{0.49\linewidth}
                \centering
                \includegraphics[width = 0.49\linewidth]{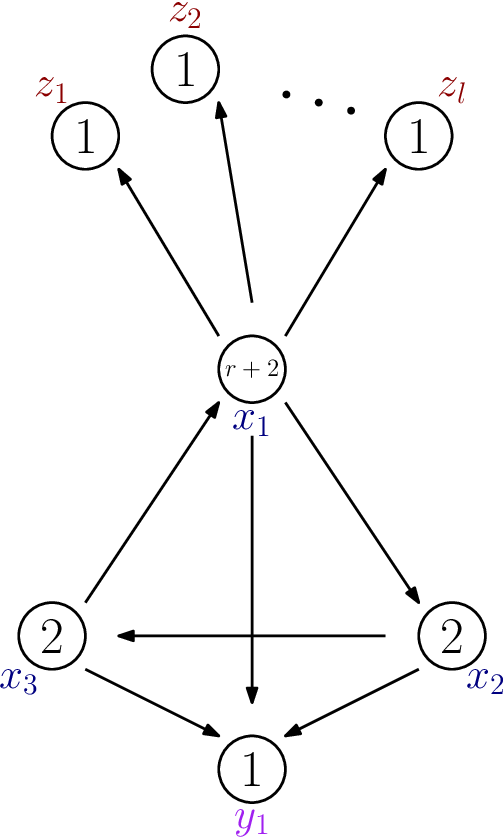}
                \subcaption{\((\mc G_{1,l,r}, \w_{1,l,r})\).}
            \end{subfigure}
            \caption{Two graphs described by the conditions in \cref{special-wographs}.}
            \label{fig:chunk-4.4-exs}
        \end{figure}
    \end{chunk}

    \begin{proposition} \label{prop:dd_complete_graph}
            For any $t,l,r$, the weighted oriented graph $(\mc G_{t,l,r}, \w_{t,l,r})$ has $3 + t + l$ vertices and satisfies $\depth (\mc G_{t,l,r}, \w_{t,l,r}) = 0$ and $\dim (\mc G_{t,l,r}, \w_{t,l,r}) = 1+l$.
    \end{proposition}

    \begin{proof}
 That $\depth (\mc G_{t,l,r}, \w_{t,l,r}) = 0$ follows from \cref{thm:depth-zero-wo-graphs}, since the subgraph of $(G_{t,l,r}, \w_{t,l,r})$ consisting of just the cycle with vertices $x_1, x_2, x_3$, and the edges between $x_1$ and the $y_i$ and $z_j$, is a maximal pseudo-forest satisfying the conditions listed in \cref{thm:depth-zero-wo-graphs}.

Now, for the dimension statement, first note that $\dim(\mc G_{t,0,r}, \w_{t,0,r}) = 1$ by \cref{thm:dim_complete_graph}. By definition, the height of $I(\mc G_{t,0,r}, \w_{t,0,r})$ does not depend on whether it is viewed as an ideal in $\Bbbk[x_1, x_2, x_3, y_1, \dots, y_t]$ or $\Bbbk[x_1, x_2, x_3, y_1, \dots, y_t, z_1, \dots, z_l]$. We thus have that \[\height I(G_{t,l,r}, \w_{t,l,r}) \geq \height I(G_{t,0,r}, \w_{t,0,r})=  (3+t)-1 = 2+t.\] On the other hand, \[I(G_{t,l,r}, \w_{t,l,r}) \subseteq (x_1, x_2, y_1, \dots, y_t)\] which shows that $\height I(G_{t,l,r}, \w_{t,l,r}) \leq 2 + t$. Hence $\height (G_{t,l,r}, \w_{t,l,r}) = 2 + t$, and since the graph $(G_{t,l,r}, \w_{t,l,r})$ has $3+t+l$ vertices, it follows that $\dim (G_{t,l,r}, \w_{t,l,r}) = (3+t+l)-(2+t) = 1 + l$.    \end{proof}

\begin{proposition} \label{prop:reg-complete-graphs-with-leaves}
    For all $t, l, r \geq 0$, $\reg (\mc G_{t,l,r}, \w_{t,l,r}) = 3 + r$. 
\end{proposition}

\begin{proof}
We rewrite the monomial generators of $(\mc G_{t,l,r}, \w_{t,l,r})$ as follows:
\[I(\mc G_{t,l,r},\w_{t,0,r}) = (x_1^{2+r}x_2, x_2^2x_3,x_3^2x_1) + \sum_{j=1}^t(x_1y_j,x_2y_j,x_3y_j) +\]\[ (\{y_iy_j \colon i<j\} + \{x_1 z_j: 1 \leq j \leq l\})\]\[= (x_1^{2+r}x_2, x_2^2 x_3, x_3^2 x_1, x_1 z_1, \dots, x_1 z_l) + \sum_{j=1}^t (x_1 y_j, x_2 y_j, x_3y_j, y_1 y_j, y_2 y_j, \dots, y_{j-1} y_j)\]
We induce on $t$ (with $l$ and $r$ fixed). The base $t = 0$ follows directly from \cref{cor:regularity-pseudo-forest}. For the induction step, suppose that the proposition holds for some $l \geq 0$. Then note that 
\[I(\mc G_{t+1, l, r}, \w_{t+1,l,r}) = I(\mc G_{t,l,r}, \w_{t,l,r}) +  y_{t+1}(x_1, x_2, x_3, y_1, \dots, y_t)\]
is a Betti splitting by \cref{thm:betti-splittings} since the latter ideal is an edge ideal of a star graph, thus has linear resolution (cf. \cite{Froberg}); hence, by \cref{thm:betti-splittings} again, we can compute $\reg (\mc G_{t+1,l,r}, \w_{t+1,l,r})$ as follows: 
\begin{align*}
    \reg (\mc G_{t+1,l,r}, \w_{t+1, l, r}) &= \begin{multlined}[t]
        \max \{\reg (\mc G_{t,l,r}, \w_{t,l,r}), \reg (x_1y_{t+1}, \dots, y_t y_{t+1}),\\
        \reg (I(\mc G_{t,l,r}, \w_{t,l,r}) \cap (x_1y_{t+1}, \dots, y_t y_{t+1})) - 1\}
    \end{multlined}\\
    &= \max \{3 + r, 1, \reg y_{t+1} I(\mc G_{t,l,r}, \w_{t,l,r}) - 1\}\\
    &=\max \{3+r, 1, 1 + (3+r) - 1\} = 3 + r,
\end{align*}
as desired.
\end{proof}

Now we are ready to prove the main theorem of this section, giving an explicit description of $\dd{n}$.
 
    \begin{theorem}\label{thm:dd-wo}
        For any $n\geq 2$, we have
                \[\dd{n}  = \ddUnw{n} \cup \{ (0,b) \colon 1\leq b \leq n-2 \}.\]
    \end{theorem}

    \begin{proof}
First, we will show the inclusion 
\[\dd{n}  \subseteq \ddUnw{n} \cup \{ (0,b) \colon 1\leq b \leq n-2 \}.\]
Consider a pair $(a, b) \in \dd{n}$. Then there exists a weighted oriented graph $(G, \w)$ with $\depth (G, \w) = a$ and $\dim (G, \w) = b$. Set $a' = \depth G$. By \cref{lem:dim(G w)=dim(G)}, $\dim G = \dim(G,\w) = b$, so $(a', b) \in \ddUnw{n}$. Also, by \cref{prop:depth(G w)-leq-depth(G)}, $a \leq a'$, so if $a \neq 0$, \cref{thm:Kanno} implies that $(a, b) \in \ddUnw{n}$. Now suppose that $a = 0$. By \cref{thm:depth-zero-wo-graphs} and \cref{lem:star}, we have that $b \neq n-1$: no subgraph of a star graph on three or more vertices, with the same number of vertices as the star graph, can be a maximal pseudo-forest, since such a subgraph would have to be a tree. Thus, $(a, b) \in \{(0, b) \colon 1 \leq b \leq n-2\}$. 

Now, let $G$ be any graph on $n$ vertices. The ideal of $G$ is the same as the ideal of the weighted oriented graph $(G, \w)$, where $\w$ assigns a weight of $1$ to each vertex of $G$ and the edges are arbitrarily oriented. Thus, $\ddUnw{n} \subseteq \dd{n}$. For the inclusion $\{(0, b) \colon 1 \leq b \leq n-2\} \subseteq \dd{n}$, we have from \cref{prop:dd_complete_graph} that $(\mc G_{n-2-b, b-1, 0}, \w_{n-2-b, b-1, 0})$ is a weighted oriented graph on $n$ vertices with depth equal to $0$ and dimension equal to $b$.
    \end{proof}

Next we will determine $\ddr{n}$, the set of all triples of depth, dimension, and regularity of weighted oriented graphs. First we need some results on regularity of weighted oriented cycles.

\begin{lemma} \label{lem:reg-cycles}
Let $C_n$ denote the naturally-oriented cyclic graph on $n$ vertices (where $n \geq 3$), and let $\w$ be a weight function on $C_n$ such that $\w(v) \geq 2$ for all vertices $v \in C_n$.  Then

\[\reg (C_n, \w) = (\sum_{v \in C_n} \w(v)) - n.\]

In particular, $\reg (C_n, \w) \geq n$. 
\end{lemma}

\begin{proof}
    This is a special case of \cref{cor:regularity-pseudo-forest}. 
\end{proof}

\begin{theorem} \label{thm:ddr-wo}
        For any $n \ge 2$, we have
            \begin{align*}
                \ddr{n} = (\ddUnw{n} \times \NN) \cup \{ (0,b,c) \colon 1\leq b\leq n-2,\ c\geq 3\}.
            \end{align*}
    \end{theorem}

\begin{proof} 
The inclusion 
\[\ddr{n} \subseteq (\ddUnw{n} \times \NN) \cup \{(0,b,c)\colon 1\leq b \leq n-2,\ c \geq 3\}\]
amounts to the assertion that a weighted oriented graph $(G, \w)$ with depth equal to $0$ must have regularity at least $3$. By \cref{thm:depth-zero-wo-graphs}, $(G, \w)$ must contain a naturally-oriented cycle $C_n$, with all vertices of $C_n$ having weight at least $2$. Suppose that $n$ is minimal, so that $C_n$ is induced. By the Restriction Lemma (\cite[Lemma~4.4]{HHZ04}), we then have that $\reg (G, \w) \geq \reg (C_n, \w|_{C_n})$. By \cref{lem:reg-cycles}, $\reg(C_n, \w|_{C_n}) \geq n \geq 3$, so the regularity of $(G, \w)$ must be at least $3$, as desired. 

For the inclusion $\ddUnw{n} \times \NN \subseteq \ddr{n}$, suppose that $(a, b) \in \ddUnw{n}$ and $r \in \NN$. By Theorem~\ref{thm:Kanno}, there exists a graph $G$ with $\reg (G) = 1$. By Lemma~\ref{lem:construct-new-graph}, there exists a weight function $\w$ on vertices of $G$ and an orientation on $G$ such that 
\begin{align*}
    \dim(G,\w) &= \dim G = a, \\
    \depth(G,\w) &= \depth G= b, \\
    \text{and } \reg(G,\w) &= \reg G + (r -1) = r.
\end{align*}

Thus, $(a,b,r) \in \ddr{n}$. The remaining inclusion $\{(0,b,c)\colon 1 \leq b \leq n-2,\  c \geq 3\} \subseteq \ddr{n}$ follows directly from Propositions~\ref{prop:dd_complete_graph} and  \ref{prop:reg-complete-graphs-with-leaves}: if $1 \leq b \leq n-2$ and $c \geq 3$, then $(\mc G, \w) = (\mc G_{n-2-b, b-1, c-3}, \w_{n-2-b, b-1, c-3})$ is a graph on $n$ vertices with 
\[(\dim(\mc G, \w), \depth(\mc G, \w), \reg(\mc G, \w) = (0, b, c),\]
as desired.
\end{proof}

\section{Betti table sizes of edge ideals of weighted oriented trees and connected bipartite graphs}\label{sec:betti-tables}

In this section, we determine the possible sizes of Betti tables of (connected) weighted oriented trees and bipartite graphs. Recall that (graded) Betti numbers record the shifts and ranks of free modules appearing in the minimal free resolution of $S/I$, where $I$ is a monomial ideal, and Betti table is a table of all (graded) Betti numbers. By definition, the horizontal length of the Betti table is the projective dimension of $S/I$, and the length of its is the regularity of $S/I$ (by definition of regularity). The size of a Betti table, therefore, is a pair of projective dimension and regularity.

All graphs in this section are assumed to be connected. The motivation for this section is due to the following results of Erey and Hibi who explicitly determined $\bptUnw{n}$ and $\treeUnw{n}$ for any $n\geq 4$. Correspondingly, we consider projective dimension in this section, rather than the complementary notion of depth.

\begin{theorem}[\protect{{\cite[Theorems 3.13 and 3.14]{EH21}}}] \label{thm:bpt-and-tree}
    Let $n \geq 4$. Then
    \[
    \treeUnw{n} = \{ (p,r)\in \mathbb{Z}^2\colon 1\leq r <\frac{n}{2} , \lceil \frac{n}{2} \rceil \leq p \leq n-r \} 
    \]
    and
    \[
    \bptUnw{n} = \{ (p,r)\in \mathbb{Z}^2\colon 1\leq r <\lfloor \frac{n}{2} \rfloor, \lceil \frac{n}{2} \rceil \leq p \leq n-2 \} \cup \{(n-1,1)\} \cup A_n
    \]
    where $A_n=\emptyset$ if $n$ is even, and $A_n=\{ (\lceil \frac{n}{2} \rceil, \lfloor \frac{n}{2} \rfloor) \}$ if $n$ is odd.
\end{theorem}

Building up on this result, we will compute an analog for weighted oriented trees and biparite graphs.

\begin{theorem}\label{thm:bipartite-wo}
    For all $n \geq 4$, 
    \[\bpt{n} = \{(p,r) \in \NN\colon \lceil \frac{n}{2} \rceil \leq p \leq n-1\} \cup \{(n, r) \colon r \geq 4\}.\]
\end{theorem}

\begin{proof}
Suppose that $(G, \w)$ is a weighted oriented bipartite graph, and set \begin{itemize}
    \item $p \coloneqq \pdim (G, \w)$, $\ p' \coloneqq \pdim (G)$,
    \item $r \coloneqq \reg (G, \w)$, $\ r' \coloneqq \reg (G)$.
\end{itemize}
By \cite[Observation 37]{casiday2021betti}, we have $p' \leq p$ and $r' \leq r$. By \cref{thm:bpt-and-tree}, it follows that $p \geq \lceil \frac{n}{2} \rceil$ and $r \geq 1$. If $p = n$, then $\depth (G, \w) = 0$, so by  \cref{thm:depth-zero-wo-graphs}, $G$ must contain a naturally oriented cycle $C_k$ as a subgraph, where all vertices of $C_k$ have weight at least $2$. By taking $k$ to be minimal, we can assume that $C_k$ is induced. Now, no bipartite graph can contain a subgraph of the form $C_l$, where $l$ is odd, so in particular, $G$ cannot contain $C_3$. Thus, $k \geq 4$. By \cref{lem:reg-cycles}, we then have $\reg (C_k, w|_{C_k}) \geq 4$, and so $\reg (G, \w) \geq 4$ by the Restriction Lemma (\cite[Lemma~4.4]{HHZ04}). This establishes the inclusion 
\[\bpt{n} \subseteq \{(p,r) \in \NN\colon \lceil \frac{n}{2} \rceil \leq p \leq n-1\} \cup \{(n, r) \colon r \geq 4\}.\]
For the inclusion 
\[\{(p,r) \in \NN\colon \lceil \frac{n}{2} \rceil \leq p \leq n-1\} \subseteq \bpt{n}\]
suppose that $\lceil \frac{n}{2} \rceil \leq p \leq n-1$ and $r \in \NN$. By \cref{thm:bpt-and-tree}, there exists a bipartite graph $G$ with $\pdim G = p$ and $\reg G = 1$. Then by \cref{lem:construct-new-graph}, there exists a weight function $\w$ on vertices of $G$ and an orientation on $G$ such that $\pdim (G, \w) = p$ and $\reg (G, \w) = r$. Thus, $(p, r) \in \bpt{n}$. 

For the inclusion 
\[
\{(n, r) \colon r \geq 4\} \subseteq \bpt{n},
\] 
let $r \geq 4$ and consider the weighted oriented graph in \cref{fig:thm-5.2}.  

\begin{figure}[ht!]
    \centering
    \includegraphics[width=0.5\linewidth]{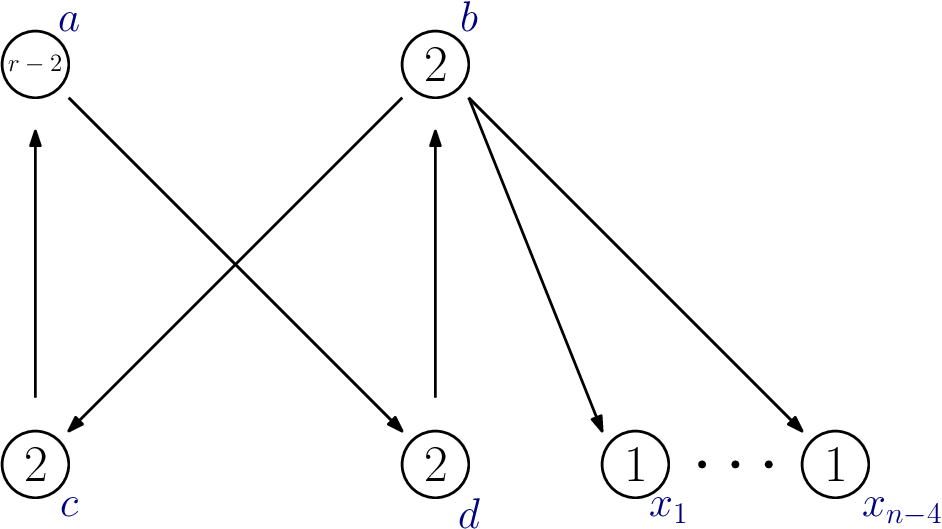}
    \caption{}
    \label{fig:thm-5.2}
\end{figure}

By \cref{cor:regularity-pseudo-forest}, this graph has projective dimension $n$, and regularity $r$, as~desired. 
\end{proof}

\begin{theorem}\label{thm:tree-wo}
    For all $n \geq 4$, 
    \[\tree{n} = \{\lceil \frac{n}{2} \rceil, \lceil \frac{n}{2} \rceil + 1, \dots, n-1\} \times \NN.\]
\end{theorem}

\begin{proof}
Suppose that $(G, \w)$ is a weighted oriented tree. Setting $p = \pdim (G, \w)$, $r = \reg (G, \w)$, $p' = \pdim (G)$, and $r' = \reg (G)$, we have that $p' \leq p$ and $r' \leq r$ by \cite[Observation 37]{casiday2021betti}. So by  \cref{thm:bpt-and-tree}, $p \geq p' \geq \lceil \frac{n}{2} \rceil$ and $r \geq r' \geq 1$.  Since a subgraph of $G$ cannot be a maximal pseudo-forest, by \cref{thm:depth-zero-wo-graphs}, $(G, \w)$ cannot have depth $0$,. Thus, $\lceil \frac{n}{2} \rceil \leq p \leq n-1$, so we see that
 \[\tree{n} \subseteq \{\lceil \frac{n}{2} \rceil, \lceil \frac{n}{2} \rceil + 1, \dots, n-1\} \times \NN.\]
Conversely, suppose that $\lceil \frac{n}{2} \rceil \leq p \leq n-1$ and $r \in \NN$. By \cref{thm:bpt-and-tree}, there is a tree $G$ with $\pdim G = p$ and $\reg G = 1$. Applying \cref{lem:construct-new-graph}, there exists a weight function $\w$ on vertices of $G$ and an orientation on $G$ such that $\pdim (G,\w)=p$ and $\reg (G,\w)= r$, as desired.
\end{proof}

\subsection*{Data availability statement} Data sharing does not apply to this article as no new data were created or
analyzed in this study.

\bibliographystyle{amsalpha}
\bibliography{references}

\end{document}